\documentclass[11pt,a4paper]{amsart}

\usepackage{amsmath}
\usepackage{amssymb}
\usepackage{amsopn}
\usepackage{amsthm}
\usepackage[latin1]{inputenc}
\usepackage[british]{babel}
\usepackage{mathtools}

\usepackage{tikz}
\newtheorem{Thm}{Theorem}

\newtheorem{Lem}{Lemma}

\theoremstyle{definition}

\newtheorem{Ex}{Example}

\theoremstyle{remark}

\newcommand{\nat}{\mathbb{N}}

\newcommand{\im}{\mathop{\rm Im}\nolimits}         % image
\newcommand{\cpx}[1]{#1_{\bullet}}                 % complex
\newcommand{\abs}[1]{| #1| }                       % degree of element

\newcommand{\multideg}{\mathop{\mathrm{deg}_{\nat^n}}}

\newcommand{\digraph}[1]{\Gamma_{#1}}
\newcommand{\morsegraph}[2]{\Gamma_{#1}^{#2}}

\newcommand{\nbhd}{\mathop{\rm nbhd}}   % neighbourhood of vertex
\newcommand{\pnbhd}{\mathop{\rm pnbhd}} % pre-neighbourhood of vertex

\newcommand{\supp}{\mathop{\rm supp}\nolimits}
\renewcommand{\epsilon}{\varepsilon}
\renewcommand{\phi}{\varphi}

\newcommand{\card}[1]{|#1|}

\newcommand{\resbasis}[2]{(#1|#2)}           % basis elt in minimal res

\begin{document}

\title[Multiplicative edge ideal resolutions]{The minimal resolution of a cointerval edge ideal is multiplicative }

\author{Emil Sk\"oldberg}
\address{School of Mathematics, Statistics and Applied Mathematics,
  National University of Ireland, Galway, Ireland}
\email{emil.skoldberg@nuigalway.ie}
\urladdr{http://www.maths.nuigalway.ie/~emil/}
\date{\today}

%\keywords{}

\begin{abstract}
  We show that the minimal resolution of the quotient of the
  polynomial algebra over a field by a cointerval edge ideal can be
  given the structure of a DG-algebra.
\end{abstract}

\maketitle

%% Introduction
\section{Introduction}

To a simple graph $G$ on the vertex set $[n] = \{1, \ldots, n\}$,
one can associate an
ideal $I_G$ in the polynomial algebra $S=k[x_1,\dots,x_n]$ over the
field $k$, by letting $I_G$ be generated by all monomials $x_ix_j$
such that $ij$ is an edge in $G$; this ideal is known as the
\emph{edge ideal} of $G$. In recent years, the study of edge ideals
has enjoyed a great deal of popularity, and several authors have
worked on relating the graph-theoretical properties of $G$ to the
algebraic properties of $I_G$.

In this paper, we study the minimal resolution of $R_G = S/I_G$ in the
case when $G$ is a cointerval graph, which is a graph that is the
complement of an interval graph. The resolution can be obtained as a
special case either of results by Chen~\cite{chen:edge_ideals}, or by
Dochtermann and Engstr\"om~\cite{dochtermann-engstrom:cointerval_ideals}.
Chen constructs the minimal resolution of $R_G$ for all complements of
chordal graphs, and since every interval graph is chordal, cointerval
ideals are covered.  Dochtermann and Engstr\"om construct the minimal
resolution of $I_G$ for all cointerval $d$-hypergraphs; our cointerval
graphs being the case of $d = 2$.

In section~\ref{sec:resolution}, we describe the resolution, in
section~\ref{sec:homotopy}
we use algebraic Morse theory to construct a contracting homotopy of
the minimal resolution $\cpx{F}$, and in section~\ref{sec:multiplicative}
we use this contracting homotopy to construct a map
$\mu: \cpx{F} \otimes_S \cpx{F} \rightarrow \cpx{F}$ which we show
gives a commutative and associative multiplication on $\cpx{F}$ making
it into a DG-algebra.

Not every cyclic module $S/I$ has the property that its minimal
resolution is multiplicative, see Avramov~\cite{avramov:obstructions}
for results on homological obstructions to the existence
of DG-algebra structures, as well as examples of ideals $I$ such that
$S/I$ does not have a multiplicative minimal resolution.
For a good survey of much of the early works on the existence and
non-existence of multiplicative structures on resolutions, see
Miller~\cite{miller:multiplicative_resolutions}.

Nevertheless, several classes of resolutions of monomial ideals have
been found to be multiplicative. Gemeda~\cite{gemeda:multiplicative}
and  Fr\"oberg~\cite{froberg:complex_constructions} have independently
shown that the Taylor resolution of a monomial ideal is
multiplicative; Peeva~\cite{peeva:borelfixed} has shown that for $I$ a
stable monomial ideal, the minimal resolution of $S/I$ is multiplicative,
and Sk\"oldberg~\cite{skoldberg:in_lin_syz} has shown the corresponding
result for matroidal ideals.

%% The resolution and its contracting homotopy.
\section{The resolution and its contracting homotopy} \label{sec:resolution}

An \emph{interval graph} is a graph whose vertices correspond to
intervals of the real line, and where two vertices are adjacent if the
corresponding intervals overlap. A \emph{cointerval graph} is the
complement of an interval graph.

\begin{Ex} \label{ex:myex}
  Consider the intervals $I_1 = [0,3]$, $I_2 = [0,1]$, $I_3 = [2,3]$
  $I_4 = [4,5]$ as depicted below:
  
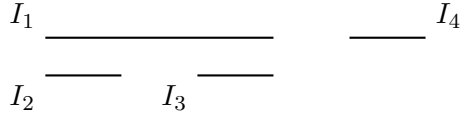
\begin{center}
    \begin{tikzpicture}
      % the nodes
      \node[] (a1) at (0,0) {};
      \node[] (b1) at (3,0) {};
      \node[] (a2) at (0,-0.5) {};
      \node[] (b2) at (1,-0.5) {};
      \node[] (a3) at (2,-0.5) {};
      \node[] (b3) at (3,-0.5) {};
      \node[] (a4) at (4,0) {};
      \node[] (b4) at (5,0) {};

      % then the lines
      \draw (a1) node [above left] {$I_1$};
      \draw [line width=0.8] (0,0) -- (3,0);
      \draw (a2) node [below left] {$I_2$};
      \draw [line width=0.8] (0,-0.5) -- (1,-0.5);
      \draw (a3) node [below left] {$I_3$};
      \draw [line width=0.8] (2,-0.5) -- (3,-0.5);
      \draw (b4) node [above right] {$I_4$} ;
      \draw [line width=0.8] (4,0) -- (5,0);
\end{tikzpicture}
  \end{center}
  The corresponding cointerval graph $G$ is thus
  
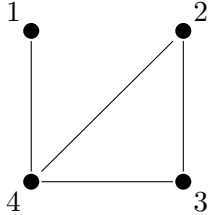
\begin{center}
    \begin{tikzpicture}
      % name the nodes
      \node[] (p1) at (-1,1) {};
      \node[] (p2) at (1,1) {};
      \node[] (p3) at (1,-1) {};
      \node[] (p4) at (-1,-1) {};

      % circles at the nodes with labels
      \fill(p1) circle (3pt) node[above left] {1} ;
      \fill(p2) circle (3pt) node[above right] {2} ;
      \fill(p3) circle (3pt) node[below right] {3};
      \fill(p4) circle (3pt) node[below left] {4};

      % draw the edges
      \draw (p1) -- (p4) (p2) -- (p3) (p2) -- (p4) (p3) -- (p4);
    \end{tikzpicture}
  \end{center}
\end{Ex}

We will now describe the minimal resolution $\cpx{F}$ of $R_G$ for $G$
a cointerval graph. Dochtermann and Engstr\"om have constructed a
polyhedral complex that supports the minimal resolution of a
cointerval $d$-hypergraph; the resolution we will study is a special
case of their construction. It is not hard to see that an interval
graph is chordal, so the resolution $\cpx{F}$ is also a special case
of Chen's construction of the minimal resolution of $R_G$ for $G$ such
that its complement $\bar{G}$ is chordal.

We will in the following assume that the vertex set is $[n]$ and that
the vertices are ordered such that if the vertex $i$ corresponds to
the interval $[a_i,b_i]$, then $a_i \leq a_j$ whenever $i<j$.

For $i$ a vertex of $G$, its \emph{neighbourhood} $\nbhd(i)$ is the
set of all vertices $j$ such that $ij \in E(G)$. Following Chen, we
also define its \emph{pre-neighbourhood} $\pnbhd(i)$ to be all $j$ in
$\nbhd(i)$ with $j<i$. We can then make the following observation.

\begin{Lem} \label{lem:pnbhd}
  Let $i$ and $j$ be vertices in $G$ with $i<j$. Then
  $\pnbhd(i) \subseteq \pnbhd(j)$.
\end{Lem}

\begin{proof}
  If $i < j$ and $k \in \pnbhd(i)$, it
  means that $[a_k,b_k]\cap[a_i,b_i] = \emptyset$, and thus
  $b_k < a_i \leq a_j$, so $k \in \pnbhd(j)$.
\end{proof}

The sets $B_i$ which will consist of the basis elements of the resolution are
now defined as follows: for the degree 0 part we let
$B_0 = \{1\}$ and for the higher degrees we let $B_d$ consist of the symbols
$\resbasis{\sigma}{\tau}$ where $\sigma, \tau \subseteq  [n]$ such
that (1) $\sigma$ and $\tau$
are disjoint and nonempty with $\card{\sigma\cup \tau}=d + 1$,
(2) $\max \sigma < \min \tau$, and (3) $\{i, \min \tau\} \in E(G)$ for all
$i \in \sigma$.

Now we can set $F_i = \bigoplus_{e \in B_i} S \cdot e$, and  describe
the differential in the complex $\cpx{F}$:
\[
\cpx{F}: \qquad 0 \longrightarrow F_r \overset{d_r}{\longrightarrow}
\cdots \overset{d_2}{\longrightarrow}
F_1 \overset{d_1}{\longrightarrow}
F_0 \overset{\epsilon}{\longrightarrow} R_G \longrightarrow 0
\]
by
\begin{align*}
d \resbasis{i}{j} &= x_ix_j \\
d \resbasis{\sigma}{\tau} &=
  \sum_{i \in \sigma} (-1)^{\alpha_1(\sigma,\tau,i)}
  x_i \resbasis{\sigma \smallsetminus i}{\tau} \\
& \phantom{=} +
  \sum_{i \in \tau} (-1)^{\alpha_2(\sigma,\tau,i)}
  x_i \resbasis{\sigma}{\tau \smallsetminus i}
\end{align*}
where
\[
\alpha_1(\sigma,\tau,i) = \card{\tau} + \card{\{j \in \sigma \mid j > i\}} \quad
\alpha_2(\sigma,\tau,i) = \card{\{j \in \tau \mid j > i\}},
\]
and where we interpret non-existent basis elements occuring in the formula as zero.
By setting
$\multideg \resbasis{\sigma}{\tau} =
\multideg \left(\prod_{i \in (\sigma \cup \tau)} x_i\right)$
we get a complex of $\nat^n$-graded modules, since it is clear
that the differential respects this grading.

\begin{Thm}[Chen, Dochtermann--Engstr\"om]
  Given a cointerval graph $G$, the complex $\cpx{F}$ defined above is
  the minimal free $\nat^n$-graded resolution of $R_G$.
\end{Thm}

\begin{proof}
  It is easy to see that the complex $\cpx{F}$ is the chain complex of
  the polyhedral complex that Dochtermann and Engstr\"om describe in
  \cite{dochtermann-engstrom:cointerval_ideals}, for the special case
  of an edge ideal of a cointerval (non-hyper)-graph.

  Alternatively, the definition of $\cpx{F}$ can be seen to agree with
  Chen's resolution, \cite[Construction 3.4]{chen:edge_ideals} by
  virtue of the conclusion of Lemma~\ref{lem:pnbhd} and the last
  remark in Chen's construction.
\end{proof}

\section{A contracting homotopy} \label{sec:homotopy}
In this section we will use methods of algebraic Morse theory to
define a contracting homotopy on the resolution. The notation we will
use is the same as in \cite{skoldberg:in_lin_syz}, whither the reader
is referred for reference.

In order to construct the contracting homotopy on $\cpx{F}$, we
consider $\cpx{F}$ to be a based complex of $k$-vector spaces with
basis elements
$x^{\alpha} \resbasis{\sigma}{\tau}$, and we will construct a Morse
matching $M$ on the directed graph $\digraph{\cpx{F}}$. To help us
show that the matching we are about to define is a Morse matching, we
partially order the elements of $B_d$ by letting
$\resbasis{\sigma_1}{\tau_1} \prec \resbasis{\sigma_2}{\tau_2}$ if
(1) $\max \tau_1 > \max \tau_2$, or (2) $\max \tau_1 = \max \tau_2$, and
$\min \sigma_1 < \min \sigma_2$.

We define three sets of edges of $\digraph{\cpx{F}}$:
$M_1$, $M_2$ and $M_3$, the union of
which will be our partial matching.

First, we let $M_1$ consist of the edges
\[
  x^{\alpha} \resbasis{\sigma}{\tau\cup j}
  \rightarrow
  x^{\alpha} x_j \resbasis{\sigma}{\tau},
  \qquad j \geq \max (\supp \alpha), j > \max(\tau).
\]

There are now two types of unmatched vertices; first we have the
vertices $x^{\alpha}$, and then the vertices
$x^{\alpha}\resbasis{\sigma}{\tau}$  where $\card{\tau} = 1$ and
$\max(\supp\alpha) \leq \max \tau$.

Next, we let $M_2$ be the edges
\[
  x^{\alpha} \resbasis{i\cup \sigma}{j}
  \rightarrow
  x^{\alpha} x_i \resbasis{\sigma}{j}
\]
satisfying
\[
i \in \nbhd(j), \quad i < \min \sigma, \quad i \leq \min (\supp \alpha \cap \nbhd(j))
\]
in the induced subgraph on the vertices $M_1^0$.
The vertices in $(M_1 \cup M_2)^0$ are then all $x^{\alpha}$
and the $x^{\alpha}\resbasis{i}{j}$ for which
$j \geq \max \supp \alpha$, and $i \leq \min (\supp \alpha \cap \nbhd(j))$,
so we let $M_3$ be the set of edges
\[
  x^{\alpha}\resbasis{i}{j} \rightarrow x^{\alpha}x_ix_j,
  \qquad \resbasis{i}{j} \text{ $\prec$-minimal such that } x_ix_j |
  x^{\alpha}x_ix_j.
\]

And, finally, we set $M = M_1 \cup M_2 \cup M_3$, and we get the
unmatched vertices $M^{0} = \{ x^{\alpha} \mid x^{\alpha} \not\in I_G \}$.

\begin{Lem}
  The set $M$ is a Morse matching on $\digraph{\cpx{F}}$.
\end{Lem}

\begin{proof}
  It is clear from construction that $M$ is a partial matching,
  so we need to show that there are no infinite paths in
  $\morsegraph{\cpx{F}}{M}$. We can see
  that if we have an elementary reduction path from
  $x^{\alpha}\resbasis{\sigma}{\tau}$ to $x^{\beta}\resbasis{\sigma'}{\tau'}$, then
  $\resbasis{\sigma'}{\tau'} \prec \resbasis{\sigma}{\tau}$
  which shows that the length of the directed path between vertices in the same
  degree is bounded.
\end{proof}

Since $M$ is a Morse matching with critical vertices $M^{0}$ concentrated
in degree 0, we get a contracting homotopy $\phi$ as in
\cite[Lemma 2]{skoldberg:morse}, which can be described in terms of
reduction paths, see J\"ollenbeck and Welker~\cite{joellenbeck_welker:morse}
and Sk\"oldberg~\cite{skoldberg:in_lin_syz}. We will next define a
$k$-linear map $c$; and then show that $c$ coincides with the contracting
homotopy $\phi$.

We will need to distinguish between three types of basis elements in
order to describe $c$:
\begin{enumerate}
\item $x^{\alpha}$. %$x^{\alpha}\cdot 1$.
\item $x^{\alpha}\resbasis{\sigma}{\tau}$ where $\card{\tau} = 1$.
\item $x^{\alpha}\resbasis{\sigma}{\tau}$ where $\card{\tau} \geq 2$.
\end{enumerate}

To the basis element $x^{\alpha}\resbasis{\sigma}{\tau}$,  we associate sets
$C_1$, $C_2$ and $C_3$ by
\begin{align*}
  C_1 &= \{i \mid i \in \supp \alpha, i > \max \tau\} \\
  C_2 &= \{i \mid i \in \supp \alpha, i < \min \sigma, \{i,\min \tau\} \in E(G)\} \\
  C_3 &= \{i \mid i \in \supp \alpha, i < \min \sigma, i < \max \supp \alpha,
  \{i, \max\supp\alpha\}\in E(G)\}
\end{align*}
and in case the corresponding set is non-empty, we let
$m_1 = \max C_1$, $m_2 = \min C_2$ and $m_3 = \min C_3$.

For the basis elements $x^{\alpha}$ we now let
\[
c(x^{\alpha}) =
\begin{dcases}
  \frac{x^{\alpha}}{x_ix_j}\resbasis{i}{j}, & \quad \text{if }
  x^{\alpha} \in I_G, \resbasis{i}{j} \text{ $\prec$-minimal such that }
  x_ix_j | x^{\alpha}, \\
  0, &\quad \text{otherwise.}
\end{dcases}
\]

Turning to the basis elements $x^{\alpha}\resbasis{\sigma}{\tau}$
where $\card{\tau} = 1$, $\tau = \{ i \}$ next, we set

\[
c(x^{\alpha} \resbasis{\sigma}{\tau}) =
\begin{cases}
 \!\begin{multlined}[b][.5\textwidth]
   \frac{x^{\alpha}}{x_{m_1}}\resbasis{\sigma}{\tau\cup m_1} \\
   +(-1)^{\card{\sigma}+1}
 \frac{x^{\alpha}x_i}{x_{m_1}x_{m_3}} \resbasis{m_3\cup \sigma}{m_1},
 \end{multlined}
 &\quad \text{if } C_1\neq \emptyset, C_3 \neq \emptyset, \\
 \!\begin{multlined}[b][.5\textwidth]
 \frac{x^{\alpha}}{x_{m_1}}\resbasis{\sigma}{\tau\cup m_1},
 \end{multlined}
 &\quad \text{if } C_1 \neq \emptyset, C_3 = \emptyset, \\
 \!\begin{multlined}[b][.5\textwidth]
 (-1)^{\card{\sigma}+1}
 \frac{x^{\alpha}}{x_{m_2}}\resbasis{m_2 \cup \sigma}{\tau},
 \end{multlined}
 &\quad \text{if } C_1 = \emptyset, C_2 \neq \emptyset, \\
 0, &\quad \text{if } C_1 = \emptyset, C_2 = \emptyset.
\end{cases}
\]

Lastly, we treat the basis elements $x^{\alpha}\resbasis{\sigma}{\tau}$
 where $\card{\tau} \geq 2$  and let
\[
c(x^{\alpha}\resbasis{\sigma}{\tau}) =
\begin{cases}
  \frac{x^{\alpha}}{x_{m_1}}\resbasis{\sigma}{\tau\cup m_1},
  &\quad \text{if } C_1 \neq \emptyset, \\
  0, &\quad \text{otherwise.}
    \end{cases}
\]

\begin{Lem}
  The map $c$ is an $\nat^n$-graded contracting homotopy of $\cpx{F}$ such that
  $c^2 = 0$ and $c(e)= 0$ for all $e \in \bigcup_{i}B_i$.
\end{Lem}

\begin{proof}
  Let $\phi$ be the homotopy we get from the Morse matching $M$; we
  shall see that $c = \phi$.

  First we look at the basis element $v = x^{\alpha}$. We have two cases,
  if $x^{\alpha} \in I_G$, then $x^{\alpha}$ is matched with
  $v' = x^{\alpha}/x_{i}x_{j} \cdot \resbasis{i}{j}$ where $\resbasis{i}{j}$
  is minimal with respect to $\prec$. There are no elementary reduction
  paths originating in $v'$, so we can conclude that in this case
  $c(v) = v' = \phi(v)$. In the case $x^{\alpha} \not\in I_G$, we have
  that $x^{\alpha} \in M^{0}$, so $c(v) = 0 = \phi(v)$.

  Next, we turn to elements $v = x^{\alpha}\resbasis{\sigma}{j}$. If
  $C_1 \neq \emptyset$, $v \in M^{-}$ and is matched with
  $v' = x^{\alpha}/x_{m_1} \cdot \resbasis{\sigma}{j m_1}$. There is an elementary
  reduction path from $v'$ to
  $v'' = x^{\alpha}x_{j}\resbasis{m_3 \cup \sigma}{m_{1}}$ precisely when
  $C_3 \neq \emptyset$. It is easy to see that there are no elementary
  reduction paths starting in $v''$, so after verifying the signs,
  we can see that $c(v) = \phi(v)$ when $C_1 \neq \emptyset$.
  If $C_1 = \emptyset$, we have that $x^{\alpha}\resbasis{\sigma}{j}\in M^{-}$
  precisely when $C_2 \neq \emptyset$, in which case $v$ is matched with
  $v' = x^{\alpha}/x_{m_2}\cdot\resbasis{m_2\cup \sigma}{j}$ and there are no
  elementary reduction paths from $v'$, so $c(v) = \phi(v)$ in this case
  as well.

  Lastly, we look at the elements $v = x^{\alpha}\resbasis{\sigma}{\tau}$
  where $\abs{\tau} \geq 2$. Here we can see that $v \in M^{-}$ precisely
  when $C_1 \neq \emptyset$, in which case $v$ is matched with
  $v' =  x^{\alpha}/x_{m_1}\cdot\resbasis{\sigma}{\tau \cup m_1}$ There are no
  elementary reduction paths from $v'$ which shows that $c(v) = \phi(v)$
  for these elements too.

  It is clear from the definition that $c$ respects the multidegree, and
  since $c(v) = 0$ for all elements in $M^{+}$, we can see that $c^2 = 0$
  and $c(e) = 0$ for all $S$-basis elements in $B_m$.

\end{proof}

%% the algebra structure on the resolution
\section{The multiplicative structure} \label{sec:multiplicative}

Now we are in a position that allows us to define the
multiplication making $\cpx{F}$ into a DGA. Just like in
\cite{skoldberg:in_lin_syz}, we are going to use the following result
in the construction.

\begin{Lem}\label{lem:maclane}
  Suppose that $\cpx{X}$ and $\cpx{Y}$ are complexes of $S$-modules,
  where $X_n = S\otimes_k V_n$ and $Y_n = S \otimes_k W_n$ for
  $k$-spaces $V_n$ and $W_n$, $n \geq 0$. Furthermore, suppose that
  $\cpx{Y}$ is acyclic, with a contracting homotopy $c$ satisfying
  $c^2 =0$. Then, every $S$-linear map $\phi_0:X_0 \longrightarrow Y_0$
  has a unique lifting to a chain map $\phi: \cpx{X} \longrightarrow \cpx{Y}$
  satisfying $\phi(V_n) \subseteq \im c$. This map is defined
  inductively by
  \[
  \phi_{n+1}(\bar{x}) = c\phi_{n}d(\bar{x}), \qquad \bar{x} \in V_{n+1}.
  \]
\end{Lem}
\begin{proof}
This is a special case of~\cite[Theorem IX.6.2]{maclane:homology}.
\end{proof}

We now let $\mu$ be the map
$\mu: \cpx{F} \otimes_{S} \cpx{F} \rightarrow \cpx{F}$ that is the
lifting of the canonical isomorphism
$\mu_0:F_0\otimes_S F_0 = S\otimes_S S \rightarrow S = F_0$
using the contracting homotopy $c$ from the previous section.
This will be our proposed product on $\cpx{F}$
so we will henceforth write $x \star y$ for $\mu(x \otimes y)$.

\begin{Lem} \label{lem:dga_basics}
  For all basis elements $x,y$ of $\cpx{F}$ we have
  \begin{enumerate}
  \item $d(x\star y) = d(x) \star y + (-1)^{\abs{x}} x \star d(y)$.
  \item $x \star y = (-1)^{\abs{x}\abs{y}} y \star x$.
  \item $1 \star x = x \star 1 = x$.
  \end{enumerate}
\end{Lem}

\begin{proof}
  Claim (1) just expresses that $\mu$ is a chain map.
  For (2), we let $\tau: \cpx{F}\otimes_S\cpx{F} \longrightarrow
  \cpx{F}\otimes_S\cpx{F}$ be defined on basis elements $x, y$ by
  $\tau(x\otimes y) = (-1)^{\abs{x}\abs{y}} y\otimes x$. Now $\mu$ and
  $\mu\circ\tau$ are
  chain maps lifting the same map in degree 0 and both mapping basis
  elements to $\im c$; so by Lemma~\ref{lem:maclane}, they must be
  equal.
  Claim (3) is proven by induction on the degree of $x$.
\end{proof}

Let us now define a map $\partial: F_{n} \rightarrow F_{n-1}$, $n \geq 1$,
by
\begin{align*}
\partial \resbasis{i}{j} &= x_ix_j \\
\partial \resbasis{\sigma}{\tau} &=
  x_{\max(\tau)} \resbasis{\sigma}{\tau \smallsetminus \max(\tau)} \\
& \phantom{=} -
  (-1)^{\abs{\tau} + \abs{\sigma}}
  x_{\min(\sigma)} \resbasis{\sigma \smallsetminus \min(\sigma)}{\tau},
\end{align*}
again treating any non-existent basis elements occurring as zero.

Its usefulness comes from that we can replace the real differential $d$
by $\partial$ when reasoning about the multiplication, as the following
lemma shows.

\begin{Lem}\label{lem:partial_diff}
  For basis elements $\resbasis{\sigma_1}{\tau_1}$,
  $\resbasis{\sigma_2}{\tau_2}$ we have
  \[
  c(d(\resbasis{\sigma_1}{\tau_1}) \star \resbasis{\sigma_2}{\tau_2})
  =
  c(\partial(\resbasis{\sigma_1}{\tau_1}) \star \resbasis{\sigma_2}{\tau_2})
  \]
\end{Lem}

\begin{proof}
  Consider the difference
  \[
  c(d(\resbasis{\sigma_1}{\tau_1}) \star \resbasis{\sigma_2}{\tau_2})
  -
  c(\partial(\resbasis{\sigma_1}{\tau_1}) \star \resbasis{\sigma_2}{\tau_2})
  =
  c((d(\resbasis{\sigma_1}{\tau_1}) - \partial(\resbasis{\sigma_1}{\tau_1})) \star \resbasis{\sigma_2}{\tau_2})
  \]
  A term occuring in
  $d(\resbasis{\sigma_1}{\tau_1}) - \partial(\resbasis{\sigma_1}{\tau_1})$
  is either of the form $x_i\resbasis{\sigma_1\smallsetminus i}{\tau_1}$
  where $i > \min(\sigma)$ or
  $x_i\resbasis{\sigma_1}{\tau_1\smallsetminus i}$ where $i < \max(\tau_1)$.

  Now assume that
  $x_k\resbasis{\sigma_3}{\tau_3}$ occurs in a product
  $\resbasis{\sigma_1\smallsetminus i}{\tau_1} \star
  \resbasis{\sigma_2}{\tau_2}$ or
  $\resbasis{\sigma_1}{\tau_1\smallsetminus i} \star
  \resbasis{\sigma_2}{\tau_2}$, and assume further that
  $c(x_ix_k\resbasis{\sigma_3}{\tau_3}) \neq 0$. This means that
  either (i), $i \geq k$ and $i > \max(\tau_3)$, which implies that
  $i = \max(\tau_1)$, or (ii), $\abs{\tau_3} = 1$, and
  $i < \min(\sigma_3)$.
  Now, in case (ii), if $k < i$, we would have that $\tau_3 = \{t\}$ where
  $t = \max (\tau_1 \cup \tau_2)$, so by Lemma~\ref{lem:pnbhd} it would
  be the case that
  $k \in \pnbhd(t)$, but then $x_k\resbasis{\sigma_3}{\tau_3} \in M^{-}$
  which contradicts that $x_k\resbasis{\sigma_3}{\tau_3} \in M^{+}$.
  Thus $i \leq k$, so $i = \min \sigma_1$.
\end{proof}

We will now give an explicit description of the multiplication
in the simplest non-trivial case.

\begin{Lem} \label{lem:mult_bottom_degree}
  Let $\resbasis{s_1}{t_1}$ and $\resbasis{s_2}{t_2}$ be basis elements
  of degree 1 in $\cpx{F}$. Then  %If $i \leq k$, we have
  \[
  \resbasis{s_1}{t_1} \star \resbasis{s_2}{t_2} =
  \begin{cases}
    x_{s_1}\resbasis{s_2}{t_2t_1} + x_{t_2}\resbasis{s_1s_2}{t_1}
    & \quad t_1 > t_2, s_1 < s_2, \\
    x_{s_1}\resbasis{s_2}{t_2t_1}
    & \quad t_1 > t_2, s_1 = s_2, \\
    x_{s_1}\resbasis{s_2}{t_2t_1} - x_{t_2}\resbasis{s_2s_1}{t_1}
    & \quad t_1 > t_2, s_1 > s_2, \\
    x_{t_1}\resbasis{s_1s_2}{t_2}
    & \quad t_1 = t_2, s_1 < s_2, \\
    0 & \quad t_1 = t_2, s_1 = s_2, \\
    -x_{t_2}\resbasis{s_2s_1}{t_1}
    & \quad t_1 = t_2, s_1 > s_2, \\
    x_{t_1}\resbasis{s_1s_2}{t_2} - x_{s_2}\resbasis{s_1}{t_1t_2}
    & \quad t_1 < t_2, s_1 < s_2, \\
    - x_{s_2}\resbasis{s_1}{t_1t_2}
    & \quad t_1 < t_2, s_1 = s_2, \\
    - x_{s_2}\resbasis{s_1}{t_1t_2} - x_{t_1}\resbasis{s_2s_1}{t_2}
    & \quad t_1 < t_2, s_1 > s_2.
  \end{cases}
  \]
\end{Lem}

\begin{proof}
  By the definition of the product map
  \[
  \resbasis{s_1}{t_1} \star \resbasis{s_2}{t_2} =
  c(x_{s_1}x_{t_1}\resbasis{s_2}{t_2}) - c(x_{s_2}x_{t_2} \resbasis{s_1}{t_1})
  \]
  from which the statement follows by using the definition of $c$.
\end{proof}

\begin{Lem} \label{lem:cell_containment}
  Let $\resbasis{\sigma_1}{\tau_1}$ and $\resbasis{\sigma_2}{\tau_2}$
  be basis elements of $\cpx{F}$. If $x_k\resbasis{\sigma_3}{\tau_3}$
  occurs in $\resbasis{\sigma_1}{\tau_1} \star \resbasis{\sigma_2}{\tau_2}$,
  then $\sigma_3 \subseteq \sigma_1 \cup \sigma_2$ and
  $\tau_3 \subseteq \tau_1 \cup \tau_2$.
\end{Lem}

\begin{proof}
  We use induction over
  $d = \deg \resbasis{\sigma_1}{\tau_1} + \deg \resbasis{\sigma_2}{\tau_2}$.
  If $d = 2$, the claim follows from Lemma~\ref{lem:mult_bottom_degree}.
  If $d \geq 3$,
  we look at
  $c(\partial(\resbasis{\sigma_1}{\tau_1})\star\resbasis{\sigma_2}{\tau_2})$.
  In the case of $\deg\resbasis{\sigma_1}{\tau_1} = 1$, so
  $\resbasis{\sigma_1}{\tau_1} = \resbasis{i}{j}$, it is equal to
  $c(x_ix_j\resbasis{\sigma_2}{\tau_2})$. The only terms that could
  occur in $c(x_ix_j\resbasis{\sigma_2}{\tau_2})$ are
  $x_j\resbasis{i \cup \sigma_2}{\tau_2}$,
  $x_i\resbasis{\sigma_2}{\tau_2\cup j}$ and
  $x_l\resbasis{i \cup\sigma_2}{\tau_2\smallsetminus l \cup j}$, all
  of which satisfy the statement of the lemma.

  Next we turn to the case of $\deg\resbasis{\sigma_1}{\tau_1} \geq 2$,
  and, letting $s = \min \sigma_1, t = \max \tau_1$,  we consider
  \[
  c(\partial(\resbasis{\sigma_1}{\tau_1})\star\resbasis{\sigma_2}{\tau_2} =
  c(x_{t} \resbasis{\sigma_1}{\tau_1 \smallsetminus t}\star\resbasis{\sigma_2}{\tau_2})
  \pm
  c(x_{s} \resbasis{\sigma_1 \smallsetminus s}{\tau_1}\star\resbasis{\sigma_2}{\tau_2}).
  \]
  First, suppose that $x_l \resbasis{\sigma_4}{\tau_4}$ occurs in
  $\resbasis{\sigma_1}{\tau_1 \smallsetminus t}\star\resbasis{\sigma_2}{\tau_2}$,
  then the only terms that can occur in
  $c(x_{t}x_l \resbasis{\sigma_4}{\tau_4})$ are
  $v_1 = x_l \resbasis{\sigma_4}{\tau_4\cup t}$ and
  $v_2 = x_m \resbasis{l \cup \sigma_4}{t}$. Note that if $v_2$ occurs,
  we must have $l = \min(\sigma_1 \cup \sigma_2 \cup \tau_1 \cup \tau_2)$,
  so $l \in \sigma_1 \cup \sigma_2$.
  Next suppose that
  $x_l \resbasis{\sigma_4}{\tau_4}$ occurs in
  $\resbasis{\sigma_1 \smallsetminus s}{\tau_1}\star\resbasis{\sigma_2}{\tau_2}$,
  the only term that can occur in   $c(x_{s}x_l \resbasis{\sigma_4}{\tau_4})$ is then
  $v_3 = x_l \resbasis{s\cup\sigma_4}{\tau_4}$.
  By induction, we have in both cases that
  $\sigma_4 \subseteq \sigma_1 \cup \sigma_2$ and
  $\tau_4 \subseteq \tau_1 \cup \tau_2$, so
  all of  $v_1$, $v_2$ and $v_3$ satisfy the conclusions of the lemma, and since the
  multiplication is graded commutative, the above argument also shows that
  all terms occuring in
  $c(\resbasis{\sigma_1}{\tau_1}\star\partial(\resbasis{\sigma_2}{\tau_2}))$
  also satisfy the conclusion of the lemma, and thus, by invoking Lemma~\ref{lem:partial_diff}, we have shown that all
  terms occuring in
  $\resbasis{\sigma_1}{\tau_1}\star\resbasis{\sigma_2}{\tau_2}$ satisfy
  the conclusion of the lemma, and we are done.
\end{proof}

\begin{Lem} \label{lem:cell_cardinality}
  For basis elements $\resbasis{\sigma_1}{\tau_1}$,
  $\resbasis{\sigma_2}{\tau_2}$ in
  $\cpx{F}$, we have that if $x_k\resbasis{\sigma_3}{\tau_3}$ occurs in the
  product $\resbasis{\sigma_1}{\tau_1} \star \resbasis{\sigma_2}{\tau_2}$, then
  $\max \tau_3 = \max (\tau_1 \cup \tau_2)$
  and
  $\card{\tau_3} \geq \card{\tau_1} + \card{\tau_2} - 1$.
\end{Lem}

\begin{proof}
  For the first claim we observe that if
  $\max \tau_3 \neq \max (\tau_1 \cup \tau_2)$, then $k =  \max(\tau_1 \cup \tau_2)$,
  which would imply that $x_k\resbasis{\sigma_3}{\tau_3} \in M^{-}$, but
  since $x_k\resbasis{\sigma_3}{\tau_3} \in \im c$, we know that
  $x_k\resbasis{\sigma_3}{\tau_3}\in M^{+}$.

  For the second claim, we observe that
  \[
  \card{\sigma_3} + \card{\tau_3} =
  \card{\sigma_1} + \card{\tau_1} +
  \card{\sigma_2} + \card{\tau_2} - 1
  \]
  so by Lemma~\ref{lem:cell_containment}
  \begin{align*}
    \card{\tau_3} &=
    (\card{\sigma_1} + \card{\sigma_2} - \card{\sigma_3}) +
    \card{\tau_1} + \card{\tau_2} - 1 \\
                 & \geq
                 \card{\tau_1} + \card{\tau_2} - 1.
  \end{align*}
\end{proof}

\begin{Lem} \label{lem:im_c}
  Let $\resbasis{\sigma_1}{\tau_1}$, $\resbasis{\sigma_2}{\tau_2}$ and
  $\resbasis{\sigma_3}{\tau_3}$ be basis elements of $\cpx{F}$, then
  $\resbasis{\sigma_1}{\tau_1} \star \big(\resbasis{\sigma_2}{\tau_2} \star
  \resbasis{\sigma_3}{\tau_3}\big) \in \im c$.
\end{Lem}

\begin{proof}
  Suppose $x_j\resbasis{\sigma_4}{\tau_4}$ occurs in
  $\resbasis{\sigma_2}{\tau_2}\star\resbasis{\sigma_3}{\tau_3}$ and
  furthermore that $x_k\resbasis{\sigma_5}{\tau_5}$ occurs in
  $\resbasis{\sigma_1}{\tau_1}\star\resbasis{\sigma_4}{\tau_4}$.

  Suppose that $x_jx_k\resbasis{\sigma_5}{\tau_5} \not\in \im c$. Then
  we must have that $c(x_jx_k\resbasis{\sigma_5}{\tau_5}) \neq 0$,
  which can only happen if $c(x_j\resbasis{\sigma_5}{\tau_5}) \neq 0$.
  Since
  $\max \tau_5 = \max(\tau_1\cup\tau_4) = \max(\tau_1\cup\tau_2\cup\tau_3)$,
  we have that $j < \min \sigma_5$ and that $\card{\tau_5} = 1$, so by
  lemma~\ref{lem:cell_cardinality} this means that $\card{\tau_i} = 1$
  for $1 \leq i \leq 4$ and we can define $m_i$, $1 \leq i \leq 5$
  by $\{m_i\} = \tau_i$.

  It cannot be the case that $k < \min \sigma_5$, since that would imply
  that one of $km_1$, $km_2$, or $km_3$ is in
  $E(G)$, and thus, by Lemma~\ref{lem:pnbhd} and
  Lemma~\ref{lem:cell_containment}, that $km_5 \in E(G)$ which would mean that
  $x_k\resbasis{\sigma_5}{\tau_5} \in M^{-}$. This means that
  $\min \sigma_5 \leq \min \sigma_4$.

  Therefore we can conclude that $j < \min \sigma_5 \leq \min
  \sigma_4$,
  so we have that one of $jm_2$ and $jm_3$ is in $E(G)$, so
  $jm_4 \in E(G)$, and $x_j\resbasis{\sigma_4}{\tau_4} \in M^{-}$ which
  contradicts that $x_j\resbasis{\sigma_4}{\tau_4} \in \im c$ and thus is in
  $M^{+}$.
\end{proof}

\begin{Thm}\label{thm:dga}
  For a cointerval graph $G$, the minimal resolution $\cpx{F}$ of
  $I_G$ is a DGA over $S$.
\end{Thm}

\begin{proof}
  Lemma~\ref{lem:dga_basics} gives that the proposed multiplication
  has a unit, satisfies the Leibniz rule and is graded commutative.
  It thus remains to see associativity. To this end we look at the two
  chain maps
  \[
  \mu \circ (\mu \otimes 1), \mu \circ (1 \otimes \mu):
  \cpx{F} \otimes_S \cpx{F} \otimes_S \cpx{F} \longrightarrow \cpx{F}.
  \]
  Since they agree in degree 0; Lemma~\ref{lem:maclane} tells us that
  it is enough to show that the images of basis elements under both
  maps lie in $\im c$.

  Let $e_1$, $e_2$ and $e_3$ be
  basis elements of $\cpx{F}$. If any of them is of degree zero, and
  thus equal to $1$, it is by Lemma~\ref{lem:dga_basics} obvious that
  $e_1\star (e_2 \star e_3)$ and  $(e_1\star e_2) \star e_3$, lie in
  $\im c$, so let us
  assume that this is not the case. Then, by Lemma~\ref{lem:im_c} we
  know that
  \[
  e_1 \star (e_2 \star e_3) \in \im c.
  \]
  and that
  \[
  (e_1 \star e_2) \star e_3 =
  (-1)^{\abs{e_3}(\abs{e_1}+\abs{e_2})} e_3 \star (e_1 \star e_2)
  \in \im c.
  \]
\end{proof}

We conclude by calculating the full DGA-structure on the resolution
of the graph from Example~\ref{ex:myex}.
\begin{Ex}
  Continuing with our example, we have the following $S$-basis elements
  in the resolution:
  \begin{center}
  \begin{tabular}{l|l}
    Degree & Basis elements \\
    \hline
    1 & $\resbasis{1}{4}$, $\resbasis{2}{3}$, $\resbasis{2}{4}$, $\resbasis{3}{4}$ \\
    2 & $\resbasis{12}{4}$, $\resbasis{13}{4}$, $\resbasis{23}{4}$, $\resbasis{2}{34}$\\
    3 & $\resbasis{123}{4}$ \\
  \end{tabular}
  \end{center}
  We can now get the products of elements of degree 1 from
  Lemma~\ref{lem:mult_bottom_degree}. Since the product is graded commutative
  we have zeros on the diagonal, and elements below the diagonal are the
  negative of their transposes, so we do not include them in the table.

  \begin{center}
    \begin{tabular}{c|llll}
      $\star$ & $\resbasis{1}{4}$ & $\resbasis{2}{3}$ & $\resbasis{2}{4}$ & $\resbasis{3}{4}$ \\
      \hline
      $\resbasis{1}{4}$ & & $x_{1} \resbasis{2}{34} + x_{3} \resbasis{12}{4}$ & $x_{4} \resbasis{12}{4}$ & $x_{4} \resbasis{13}{4}$ \\
      $\resbasis{2}{3}$ & & & $-x_{2} \resbasis{2}{34}$  & $x_3 \resbasis{23}{4} - x_{3} \resbasis{2}{34}$ \\
      $\resbasis{2}{4}$ & & & & $x_{4} \resbasis{23}{4}$ \\
      $\resbasis{3}{4}$ & & & &
    \end{tabular}
  \end{center}
  Next we can compute the products of an element of degree 1 with an
  element of degree 2. From the $\nat^n$-homogeneity of $\star$ it follows
  that $\resbasis{\sigma_1}{\tau_1}\star\resbasis{\sigma_2}{\tau_2} = 0$
  if $\card{(\sigma_1 \cup \tau_1) \cap (\sigma_2 \cup \tau_2)} \geq 2$,
  so we leave those entries blank in the table, and only include entries
  which need to be calculated.

  \begin{center}
    \begin{tabular}{c|cccc}
      $\star$ & $\resbasis{12}{4}$ & $\resbasis{13}{4}$ & $\resbasis{23}{4}$ & $\resbasis{2}{34}$ \\
      \hline
      $\resbasis{1}{4}$ & & & $-x_{4} \resbasis{123}{4}$ & 0 \\
      $\resbasis{2}{3}$ & 0 & $x_{3} \resbasis{123}{4}$ & &  \\
      $\resbasis{2}{4}$ & & $x_{4} \resbasis{123}{4}$ & & \\
      $\resbasis{3}{4}$ & $-x_{4} \resbasis{123}{4}$ & & & \\
    \end{tabular}
  \end{center}
\end{Ex}

\bibliographystyle{amsalpha}
\bibliography{bibfil}

\end{document}